\documentclass[a4paper,12pt]{article}
\usepackage{amsmath,amssymb,amsthm,graphics,latexsym,amsfonts}
\usepackage{fancyhdr}
\usepackage{color}
\usepackage{graphics}

\title{\Large  A  $\overrightarrow{P_{3}}$-decomposition of tournaments and bipartite digraphs
\thanks{Research supported by NSFC (No. 11571294) and by Xinjiang Talent Youth Project (No. 2013721012)}}
\author{ {Fangxia Wang, Baoyindureng Wu \footnote{Corresponding author.
Email: baoywu@163.com (B. Wu) }, Xinhui An}\\
\small  College of Mathematics and System Sciences, Xinjiang
University \\ \small  Urumqi, Xinjiang 830046, P.R.China \\}
\date{}

\newtheorem{theorem}{Theorem}[section]
\newtheorem{lemma}[theorem]{Lemma}
\newtheorem{corollary}[theorem]{Corollary}
\newtheorem{proposition}[theorem]{Proposition}

\usepackage{indentfirst}

\begin{document}
\maketitle {\small \noindent{\bfseries Abstract} A
$\overrightarrow{P_{3}}$-decomposition of a directed graph $D$ is a
partition of the arcs of $D$ into directed paths of length $2$. In
this paper, we give a characterization for a tournament and a
bipartite digraph admitting a
$\overrightarrow{P_{3}}$-decomposition. This solves a problem posed
by Diwan ($\overrightarrow{P_{3}}$-decomposition of directed graphs,
Discrete Appl. Math., http://
dx.doi.org/10.1016/j.dam.2016.01.039.).


\noindent{\bfseries Keywords:} Path-decomposition; Tournaments;
Bipartite digraphs; Line graphs; Perfect matchings

\section {\large Introduction}

A {\it decomposition} of a graph $G$ is a family $\mathcal{F}$ of
edge-disjoint subgraphs of $G=(V(G), E(G))$ such that
$\cup_{F\in\mathcal{F}}E(F)=E(G)$. In particular, a
$P_3$-decomposition of a graph $G$ is a partition of the edge set of
$G$ into paths of length 2.

The same notion of decomposition applies to directed graphs as well,
where each arc of $D$ is contained in exactly one element of the
decomposition. A \emph{$\overrightarrow{P_{3}}$-decomposition} of a
directed graph $D$ is a partition of the arcs of $D$ into directed
paths of length 2. A classical result on decomposition, stated as
follows, is due to Kotzig \cite{K}.

\begin{theorem} A connected graph $G$ has a $P_3$-decomposition if and only if
the size of $G$ is even.

\end{theorem}

However, no characterization of directed graphs that admit a
$\overrightarrow{P_3}$-decomposition is known in general. A directed
graph is said to be \emph{symmetric} if for every pair of distinct
vertices $u, v$, there is an arc from $u$ to $v$ if and only if
there is an arc from $v$ to $u$. In other words, the directed graph
is obtained by replacing each edge in an undirected graph by two
oppositely directed arcs. Diwan \cite{D} gave a characterization of
symmetric directed graphs that do not admit a
$\overrightarrow{P_3}$-decomposition.

\vspace{2mm} Let $G$ be a graph. The line graph of $G=(V(G), E(G))$,
denoted by $L(G)$, is the graph whose vertex set is $E(G)$, in which
two vertices are adjacent if and only if two edges have an end in
common in $G$. Diwan \cite{D} introduced a similar notion for a
digraph $D$ in view of Proposition 1.2 below. We adopt the same
notation of the line graph. However, it is defined for a digraph.


\vspace{3mm} \noindent{\bf Definition 1.1.} For a digraph $D=(V(D),
A(D))$, the line graph $\emph{L(D)}$ of $D$ is the graph with vertex
set $A(D)$, in which two vertices $a_1$ and $a_2$ are adjacent if
and only if $a_{1}, a_{2}$ induce a directed path of length $2$.

\vspace{2mm} A \emph{matching} in a graph is a set of pairwise
nonadjacent links. A \emph{perfect matching} is one which covers
every vertex of the graph.

\vspace{2mm} The following result is frequently used in the sequel.

\begin{proposition} A digraph $D$ has a $\overrightarrow{P_{3}}$-decomposition if and only
if $L(D)$ has a perfect matching.
\end{proposition}

By the previous result, the main tools we will use are several
classical results in matching theory. For a graph $G$, $c(G)$,
$c_o(G)$ and $i(G)$ denote the number of components, the number of
odd components, and the number of isolated vertices of $G$,
respectively.

\begin{theorem} (Tutte \cite{Tut}) A graph $G$ has a perfect matching if and only if $c_o(G-S)\leq |S|$ for all $S\subseteq
V(G)$.
\end{theorem}

For a graph $G$, $H$ is called a {\em fractional perfect matching}
of $G$ if $H$ is a spanning subgraph of $G$ such that each component
of $H$ is $K_2$ or a cycle.

\begin{theorem} (Tutte \cite{Tutte})  A graph $G$ has a fractional perfect matching if and only if
$i(G-S)\leq |S|$ for any $S\subseteq V(G)$.
\end{theorem}

For a subset $S\subseteq V(G)$, $N(S)$ denotes the set of vertices
which are adjacent to vertices of $S$ in $G$.

\begin{theorem} (Hall \cite{H}) A bipartite graph $G[X, Y]$ has a perfect matching which covers every
vertex in $X$ if and only if  $|N(S)|\geq|S|$ for all $S\subseteq
X$.
\end{theorem}

\begin{corollary} A bipartite graph $G=G[X, Y]$ has a perfect matching
if and only if $|X|=|Y|$ and $|N(S)|\geq|S|$ for all $S\subseteq X$.
\end{corollary}

In this paper, we give a complete characterization for a tournament
and a bipartite digraph admitting a
$\overrightarrow{P_{3}}$-decomposition.

\section {\large Preparations}

\subsection{\large Basic terminologies and notations}


Let $G=(V(G), E(G))$ be a graph. The {\it order} and {\it size} of
$G$ are $|V(G)|$ and $|E(G)|$, respectively. We say $G$ is trivial
if its order is 1, and nontrivial otherwise. If the size of $G$ is
0, then $G$ is called an {\it empty} graph. Let $X\subseteq V(G)$.
The \emph{induced subgraph} of $G$ induced by $X$, denoted by
$G[X]$, is the subgraph of $G$ whose vertex set is $X$ and whose
edge set consists of all edges of $G$ which have both ends in $X$.

Let $X, Y\subseteq V(G)$. We denote by $E[X,Y]$ the set of edges of
$G$ with one end in $X$ and the other end in $Y$. When
$Y=V(G)\setminus X$, the set $E[X,Y]$ is called the {\it edge cut}
of $G$ associated with $X$, and is denoted by $\partial(X)$. So, a
graph is connected if and only if $\partial(X)\neq \emptyset$ for
every nonempty proper subset $X$ of $V(G)$.

Let $D$ be a digraph. If $a=(u, v)$ is an arc, then $a$ is said to
join $u$ to $v$; we also say that $u$ \emph{dominates} $v$. The
vertex $u$ is the tail of $a$, and the vertex $v$ its head; they are
the two ends of $a$. The vertices which dominate a vertex $v$ are
its \emph{inneighbours}, those which are dominated by the vertex its
\emph{outneighbours}. These sets are denoted by $N^-_{D}(v)$ and
$N^+_{D}(v)$, respectively.

We say that $D$ is {\it strict} if there exists no loops or parallel
arcs (arcs with the same head and the same tail) in $D$. We say that
$D$ is {\it asymmetric} if $(u,v)\in A(D)$ for any two vertices $u,
v\in V(D)$, then $(v,u)\notin A(D)$. For a set $A'\subseteq A(D)$,
the \emph{arc-induced subgraph} $D[A']$ is the subdigraph of $D$
whose arc set is $A'$ and whose vertex set consists of all ends of
arcs in $A'$. We denote by $D\backslash A'$ the spanning subdigraph
of $D$ obtained from $D$ by deleting the arcs in $A'$.

Let $X, Y\subseteq V(D)$. We denote the set of arcs of $D$ whose
tails lie in $X$ and whose heads lie in $Y$ by $A(X, Y)$, and their
number by \emph{$a(X, Y)$}. This set of arcs is denoted by $A(X)$
when $Y=X$, and their number by $a(X)$. When $Y=V\setminus X$, the
set $A(X, Y)$ is called the \emph{outcut} of $D$ associated with
$X$, and denoted by $\partial^{+}(X)$. Analogously, the set $A(Y,X)$
is called the {\it incut} of $D$ associated with $X$, and denoted by
$\partial^{-}(X)$. Accordingly, we refer to $|\partial^{+}(X)|$ as
the {\it outdegree} of $X$ and denote it by $d^+(X)$. Similarly, we
refer to $|\partial^{-}(X)|$ as the {\it indegree} of $X$ and denote
it by $d^-(X)$. Observe that
$\partial^{+}(X)=\partial^{-}(V\backslash X)$. For simplicity, let
$\partial(X)=\partial^{+}(X)\cup
\partial^{-}(X)$. In particular, if $X=\{v\}$, then
$\partial^{-}(X)$, $\partial^{+}(X)$, $\partial(X)$, $d^-(X)$, and
$d^+(X)$ are simply denoted by $\partial^{-}(v)$, $\partial^{+}(v)$,
$\partial(v)$, $d^-(v)$, and $d^+(v)$, respectively. Since the
digraph $D$ under consideration is strict,  $d^+_D(v)=|N^+_D(v)|$
$d^-_D(v)=|N^-_D(v)|$ for a vertex $v\in V(D)$.

A digraph $D$ is called {\it weakly connected} if its underlying
undirected graph is connected. We say that $D$ is {\it strongly
connected or strong} if $\partial^+(X)\neq \emptyset$ for every
nonempty proper subset $X$ of $V(D)$.

\subsection{Connectedness of $L(D)$ for a digraph $D$}

A basic question is that for a digraph $D$, when is $L(D)$
connected. To answer this question, let us define a digraph $D'$
obtained from $D$ by repeatedly modified as follows: if there exists
such a vertex $v$ with $d^+_D(v)=0$ and $d^-_D(v)=k\geq 2$, split
$v$ into $k$ vertices $v_1, \cdots, v_k$ such that an arc with head
$v$ become an arc with head $v_i$ for some $i$; if there exists such
a vertex $u$ with $d^-_D(u)=0$ and $d^+_D(u)=l\geq 2$, split $u$
into $l$ vertices $u_1, \cdots, u_l$ such that an arc with tail $u$
become an arc with tail $u_j$ for some $j$. We repeat the operation
until no such vertex $v$ or $u$ with the property as described above
in the resulting new graph.

\begin{theorem} For a strict and asymmetric digraph $D$ without isolated vertices, $L(D)$ is connected if and only if $D'$ is weakly
connected.
\end{theorem}
\begin{proof} By the definition of $D'$, $L(D)=L(D')$ and
for any vertex $v\in V(D')$, $d^+_{D'}(v)=0$ and $d^-_{D'}(v)=1$, or
$d^-_{D'}(v)=0$ and $d^+_{D'}(v)=1$, or $d^+_{D'}(v)\geq 1$ and
$d^-_{D'}(v)\geq 1$.

First we prove the necessity by contradiction. Suppose that $D'$ is
not weakly connected. Let $D_1'$ be a nontrivial component of $D'$. Since $E[A(D_1'), A(D')\setminus A(D_1')]=\emptyset$ is an edge cut
of $L(D')$, $L(D')$ is disconnected, contradicting the assumption
that $L(D')$ is connected.

Next we assume that $D'$ is weakly connected. To show that $L(D')$
is connected, we take two arcs $a=(u,v)$ and $b=(x,y)$ in $D'$. It
suffices to show that there is a path joining $a$ and $b$ in
$L(D')$.

\vspace{2mm}\noindent{\bf Case 1.} $x=v$ or $y=u$

Then $a, b$ is a path joining $a$ and $b$ in $L(D')$.

\vspace{2mm}\noindent{\bf Case 2.} $x=u$ or $y=v$.

If $x=u$, then $d^+_{D'}(u)\geq 2$. By the definition of $D'$,
$d^-_{D'}(u)\geq 1$. Take an arc $c=(z,u)\in A(D')$, where $z\in
N^-_{D'}(u)$. So, $a, c, b$ is a path joining $a$ and $b$ in
$L(D')$.

If $y=v$, one may find a path joining $a$ and $b$ in $L(D')$ by the
similar way as above.

\vspace{2mm}\noindent{\bf Case 3.} $\{u,v\}\cap \{x, y\}=\emptyset$.

By our assumption, we choose a shortest path $P=u_1, \cdots, u_k$
connecting an end of $a$ to an end of $b$ in $D'$. Without loss of
generality, $u_1=u$ and $u_k=x$. It follows that $v\notin \{u_1,
\cdots, u_k\}$ and $y\notin \{u_1, \cdots, u_k\}$ and by the
definition of $D'$, $d^+_{D'}(u_i)\geq 1$ and $d^-_{D'}(u_i)\geq 1$
for any $i$.

Let $A(P)=\partial(\{u_1, \cdots, u_k\})$. Since $D$ is strict and
asymmetric, for each $i\in\{1, \ldots, k\}$, $L(D')[\partial(u_i)]$
is a complete bipartite graph, and $L(D')[\partial(u_i)]$ and $
L(D')[\partial(u_{i+1})]$ have a common vertex. It follows that
$L(D')[\partial(u_i)\cup
\partial(u_{i+1})]$ is connected and thus
$L(D')[\cup_{i=1}^{k}(\partial(u_i)]=L(D')[A(P)]$ is connected. In
particular, $a$ and $b$ is joined by a path in $L(D')[A(P)]$.

\end{proof}

The following corollary is a direct consequence of Theorem 2.1.

\begin{corollary} Let $T$ be a tournament of order $n\geq 3$.
Then $L(T)$ is disconnected if and only if $T$ has an arc
$(u, v)$ with $d^-_{T}(u)=0$ and $d^+_{T}(v)=0$.
\end{corollary}
\begin{proof} If $T$ has an arc
$(u, v)$ with $d^-_{T}(u)=0$ and $d^+_{T}(v)=0$, then $a=(u,v)$ is an
isolated vertex of $L(T)$. Since $n\geq 3$, the order of $L(T)$ is
at least 3, and thus $L(T)$ is disconnected.

Now we show its necessity. Observe that there exists at most one
vertex $u$ with $d^-_T(u)=0$ and at most one vertex $v$ with
$d^+_T(v)=0$. Let $T'$ be the digraph obtained from $T$ as defined
in the beginning of this section. As we saw before, $L(T)=L(T')$.
One can see that if $T$ has no arc $(u, v)$ with $d^-_{T}(u)=0$ and
$d^+_{T}(v)=0$, then $T'$ is weakly connected. By Theorem 2.1,
$L(T)$ is connected.
\end{proof}

The proof of the following result is easy, and is left to the
readers.

\begin{lemma} For a digraph $D$ without isolated vertices, $L(D)$ is
an empty graph if and only if $D$ is a bipartite digraph with
bipartition $X, Y$ such that $X=\{v\in V(D)|\ d^-_D(v)=0\}$ and
$Y=\{v\in V(D)|\ d^+_D(v)=0\}$.
\end{lemma}

%


An \emph{orientation} of a simple graph is referred to as an
oriented graph. A \emph{tournament} is an orientation of a complete
graph. A tournament is called {\it transitive} if it has no directed
cycles.

\begin{theorem} For any strict and asymmetric digraph $D$ of order $n$, $c(L(D))\leq f(n)$, where
$$
f(n)=\left \{
\begin{array}{ll}
\frac{n^{2}-1}{4}, & \mbox{if $n$ is odd\
}\\
\frac{n^{2}}{4}, &\mbox{if $n$ is even\ },
\end{array}
\right.
$$
with equality if and only if either $D$ is the transitive tournament
of order 3, or $D$ is the orientation of the balanced complete
bipartite graph $G=G[X,Y]$ of order $n$ with $A(D)=A(X, Y)$.

\end{theorem}
\begin{proof} Let $G_1, \cdots, G_k$ be all components of $L(D)$.
We take a vertex $a_i\in V(G_i)$ for each $i\in\{1, \cdots, k\}$.
Let $A^*=\{a_1, \cdots, a_k \}$ be an independent set of $L(D)$. By
Lemma 2.3, $D^*=(V(D), A^*)$ is a bipartite digraph with a
bipartition $X^*, Y^*$ such that $A_{D^*}(X^*,Y^*)=A^*$. Since $D$
is strict and asymmetric, we have $|A^*|\leq |X^*||Y^*|$,
$|X^*|+|Y^*|\leq n$, and thus  $|A^*|\leq f(n)$, with equality if
and only if
\begin{equation} A^*=\{(x,y)|\ \forall x\in X^*, \forall y\in Y^*)\}\ and
\ ||X^*|-|Y^*||\leq 1.
\end{equation}

So, $c(L(D))=|A^*|\leq f(n)$.

Now assume that $c(L(D))=f(n)$. Then $|A^*|=f(n)$ and
$$A^*=\{(x,y)|\ \forall x\in X^*, \forall y\in Y^*)\} \ and\
||X^*|-|Y^*||\leq 1.$$ Therefore, $D^*=(V(D), A^*)$ is the
orientation of the balanced complete bipartite graph $G=G[X^*,Y^*]$
of order $n$ with $A(D^*)=A(X^*, Y^*)$.

\vspace{2mm}\noindent{\bf Case 1.} $\min\{|X_1^*|, |Y_1^*|\}\geq 2$

We show that $D=D^*$. Suppose that $D\neq D^*$. Since $D$ is a
strict and asymmetric digraph, there exists an arc $a$ with both its
head and tail lying in $X^*$ or $Y^*$. Without loss of generality,
let $a=(x_1, x_2)$ be such an arc. Since $|Y^*|\geq 2$, then both
$(x_2, y_1)$ and $(x_2, y_2)$ are adjacent to $(x_1,x_2)$,
contradicting the fact that $(x_2, y_1)$ and $(x_2, y_2)$ lie in the
different components of $L(D)$. This proves $D=D^*$.

\vspace{2mm}\noindent{\bf Case 2.} $\min\{|X_1^*|, |Y_1^*|\}=1$

Without loss of generality, let $|X_1^*|=1$. If $|Y_1^*|=1$, there
is nothing to be proved. If $|Y_1^*|=2$, then either $D=D^*$ or $D$
is the transitive tournament of order 3.

Conversely, assume that $D$ is the transitive tournament of order 3,
or $D$ is the orientation of the balanced complete bipartite graph
$G=G[X,Y]$ of order $n$ with $A(D)=A(X, Y)$. It is straightforward
to check that $c(L(D))=f(n)$.

\end{proof}

\section {\large Tournament}

\begin{theorem}
Let $D$ be a digraph. Then $L(D)$ has a fractional perfect matching
if and only if $$a(X, Y)\leq a(Y, X)+a(X)+a(Y)+a(Z, X)+a(Y, Z)$$ for
any partition $X, Y, Z$ of $V(D)$.
\end{theorem}

\begin{proof} To show its necessity, we assume that $L(D)$ has a fractional perfect
matching. For a partition $X, Y, Z$ of $V(D)$, let $S=A(Y,
X)\cup A(X)\cup A(Y)\cup A(Z, X)\cup A(Y, Z)$. Note that $i(L(D)-S)\geq a(X, Y)$. By
Theorem 1.4 we have  $$a(X,Y)\leq i(L(D)-S)\leq |S|=a(Y,
X)+a(X)+a(Y)+a(Z, X)+a(Y, Z).
$$

We prove the sufficiency by contradiction. Suppose that $L(D)$ has
no fractional perfect matching. By Theorem 1.4, there exists a
subset $S\subseteq V(L(D))$ such that $i(L(D)-S)>|S|$. Let $I$ be
the set of isolated vertices in $L(D)-S$. Then $|I|=i(L(D)-S)$.
Since $I$ is an independent set of $L(D)$, by Lemma 2.3 $D[I]$ is a
bipartite digraph such that $I\subseteq A(X, Y)$, where $X$ is the
set of tails of arcs in $I$ and $Y$ is the set of heads of arcs in
$I$. Let $Z=V(D)\setminus (X\cup Y)$.

Moreover, $A(Y, X)\cup A(X)\cup A(Y)\cup A(Z, X)\cup A(Y,
Z)\subseteq S$. Combining the above, we have $$a(X,Y)\geq i(L(D)-S)>
|S|\geq a(Y, X)+a(X)+a(Y)+a(Z, X)+a(Y, Z),$$ a contradiction.

\end{proof}

Diwan \cite{D} asked the characterization of tournaments with
$\overrightarrow{P_3}$-decomposition. Indeed, we show that for a
tournament $T$, $L(T)$ has a perfect matching if and only if it has
a fractional perfect matching.

\begin{theorem}
A tournament $T$ of an even size has a
$\overrightarrow{P_{3}}$-decomposition if and only if $$a(X, Y)\leq
a(Y, X)+a(X)+a(Y)+a(Z, X)+a(Y, Z)$$ for any partition $X, Y, Z$ of
$V(T)$.
\end{theorem}

\begin{proof} If $T$ has a $\overrightarrow{P_{3}}$-decomposition, then
$L(T)$ has a perfect matching. Thus, $L(T)$ has a fractional perfect
matching. By Theorem 3.1, the result follows.

We next prove the sufficiency. In order to do so, it suffices to show that $L(T)$ has a perfect matching. Suppose that $L(T)$ has no perfect matching. By Theorem 1.3, there exists a subset $S\subseteq L(T)$ such that $c_o(L(T)-S)>|S|$. We choose the subset $S$ with two additional properties: (1) $c_o(L(T)-S)-|S|$ is maximum and (2)
subject to (1), $|S|$ is as large as possible.

\vspace{2mm} \noindent{\bf Claim 1.} Each component of $L(T)-S$ is
odd.

Suppose that $L(T)-S$ has an even component $C$. Take a vertex $v$
in $C$. Let $S'=S\cup \{v\}$. It is clear that $c_o(L(T)-S')\geq
c_o(L(T)-S)+1$, and thus $c_o(L(T)-S^{'})-|S'|\geq
c_o(L(T)-S)+1-|S|-1=c_o(L(T)-S)-|S|$, contradicting the choice (2)
of $S$. This proves Claim 1.

\vspace{2mm} Let $I$ be the set of isolated vertices in $L(T)-S$. As
defined in the proof of Theorem 3.1, $T[I]$ is a bipartite digraph
such that $I\subseteq A(X, Y)$, where $X$ is the set of tails of
arcs in $I$ and $Y$ is the set of heads of arcs in $I$. Let
$Z=V(T)\setminus (X\cup Y)$. By the choice of $S$, $I=A(X, Y)$ and
$A(Y, X)\cup A(X)\cup A(Y)\cup A(Z, X)\cup A(Y, Z)\subseteq S$.

Let $H_1, \cdots, H_k$ be all nontrivial components of $L(T)-S$. By
Claim 1, they are odd. So, $c_o(L(T)-S)=i(L(T)-S)+k=a(X,Y)+k$.

For convenience, let $S_1=A(Y, X)\cup A(X)\cup A(Y)\cup A(Z, X)\cup
A(Y, Z)$ and $S_2=S\setminus S_1$. 
Let $T'=T\setminus S_1$, as shown in Figure 1. Note that
$L(T)-S=L(T')-S_2$.

\begin{center}
\scalebox{0.4}{\includegraphics{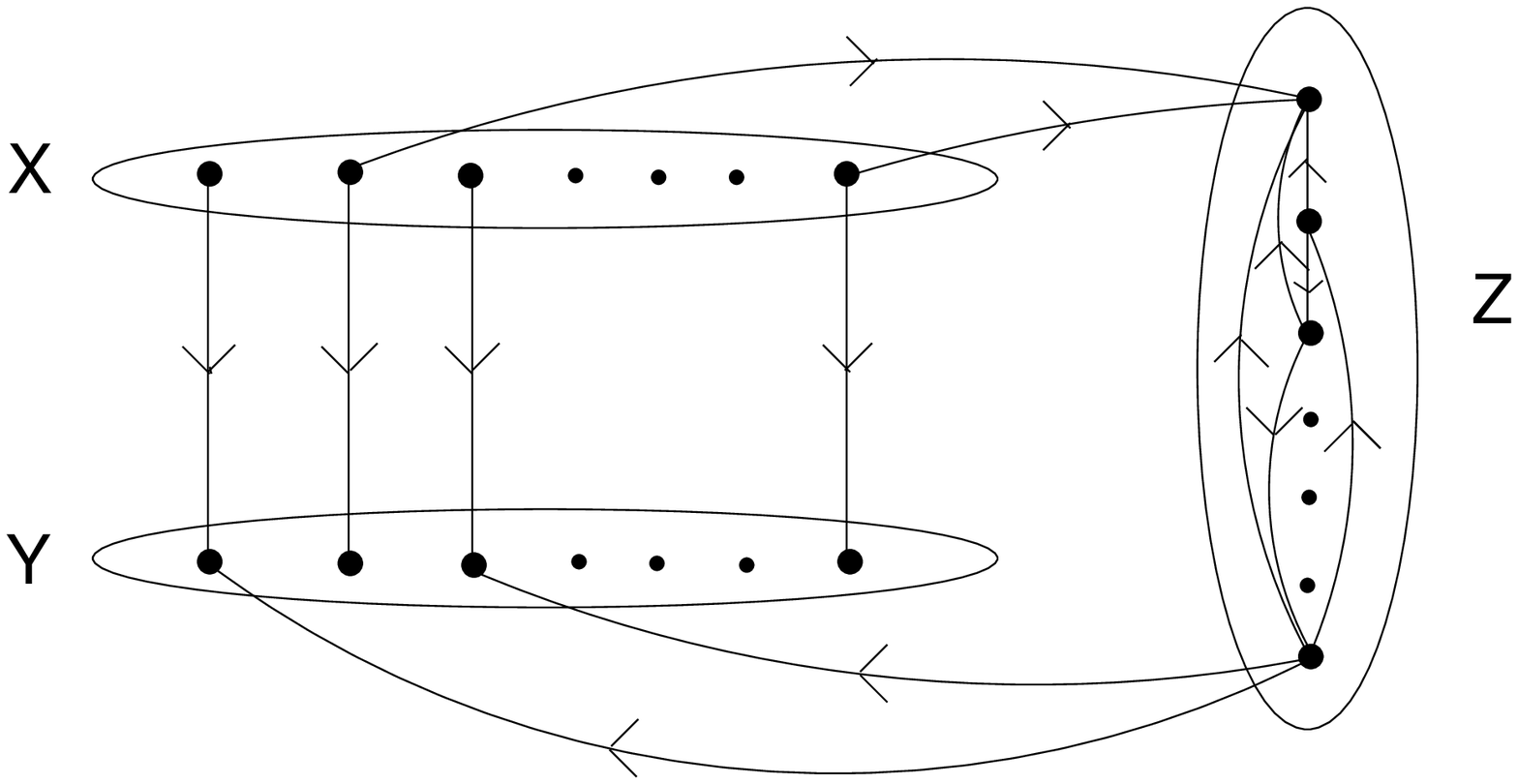}}\\
\vspace{0.4cm} Fig. 1. $T'=T\backslash S_1$
\end{center}

Since for each $i\in\{1, \ldots, k\}$, $H_i$ is nontrivial, there is
a directed path $P_i$ of length 2 in $T'$ such that $L(P_i)\subseteq
H_i$. Let $z_i$ be the central vertex of $P_i$. Clearly, $z_i\in Z$
with $d^+_{T'}(z_i)>0$ and $d^-_{T'}(z_i)>0$. Note that
$\partial(z_i)\subseteq V(H_i)$ and $\partial(z_i)\cap
V(H_j)=\emptyset$ when $i\neq j$. It follows that the arc connecting
$z_i$ and $z_j$ belongs to $S_2$ and thus $|S_2|\geq \big (^k_2
\big)$.

We consider following four cases.

\vspace{2mm} \noindent{\bf Case 1.} $k\geq 3$

\vspace{2mm} So, $|S|=|S_{1}|+|S_{2}|\geq c_o(L(T)-S)-k+\big (^k_2
\big)\geq c_o(L(T)-S)$, contradicting $c_o(L(T)-S)>|S|$.

\vspace{2mm} \noindent{\bf Case 2.} $k=2$

So, $|S|=|S_{1}|+|S_{2}|\geq c_o(L(T)-S)-2+1=c_o(L(T)-S)-1$. Since
$c_o(L(T)-S)-1\geq |S|$, $c_o(L(T)-S)-1=|S|$,  which contradicts the
fact that $|S|$ and $c_o(L(T)-S)$ have the same parity.

\vspace{2mm} \noindent{\bf Case 3.} $k=1$

\vspace{2mm} Since $c_o(L(T)-S)=i(L(T)-S)+1=a(X,Y)+1$,
$|S|=|S_1|+|S_2|\geq a(Y, X)+a(X)+a(Y)+a(Z, X)+a(Y, Z)$ and
$c_o(L(T)-S)\geq |S|+2$, we have $a(X,Y)\geq a(Y, X)+a(X)+a(Y)+a(Z,
X)+a(Y, Z)+1$, contradicting our assumption that $a(X, Y)\leq a(Y,
X)+a(X)+a(Y)+a(Z, X)+a(Y, Z)$.

\vspace{2mm} \noindent{\bf Case 4.} $k=0$

\vspace{2mm} It means that each component of $L(T)-S$ is singleton.
We have $c_o(L(T)-S)=i(L(T)-S)=a(X, Y)\leq a(Y, X)+a(X)+a(Y)+a(Z,
X)+a(Y, Z)=|S|$, contradicting that $c_o(L(T)-S)>|S|$.
\end{proof}

\begin{corollary} Let $T$ be a tournament of order $n\geq 3$.
If $T$ has an arc $(u, v)$ with $d^-_{T}(u)=0$ and $d^+_{T}(v)=0$,
then $T$ has no $\overrightarrow{P_3}$-decomposition.
\end{corollary}
\begin{proof} Let us consider the partition $X=\{u\}$, $Y=\{v\}$, $Z=V(T)\setminus \{u,
v\}$ of $V(T)$. Note that $a(X,Y)=1$,
$a(Y,X)=a(X)=a(Y)=a(Z,X)=a(Y,Z)=0$. So, $a(X,Y)\leq a(Y,
X)+a(X)+a(Y)+a(Z, X)+a(Y, Z)$ fails. By Theorem 3.2, $T$ has no
$\overrightarrow{P_3}$-decomposition.

\end{proof}

\begin{corollary} Let $T$ be a tournament of order $n\geq 3$.
If $V(T)$ has a partition $X, Y$ such that $||X|-|Y||\leq 1$ and
$A(X,Y)=\{(x,y)|\ \forall x\in X, \forall y\in Y\}$, then $T$ has no
$\overrightarrow{P_3}$-decomposition.
\end{corollary}
\begin{proof} Let $Z=\emptyset$ and we consider the partition $X$, $Y$, $Z$. Note that
$a(X,Y)=f(n)$, as defined in the statement of Theorem 2.4, where
$$
f(n)=\left \{
\begin{array}{ll}
\frac{n^{2}-1}{4}, & \mbox{if $n$ is odd\
}\\
\frac{n^{2}}{4}, &\mbox{if $n$ is even\ },
\end{array}
\right.
$$

and $a(Y,X)=a(Z, X)=a(Y, Z)=0$ and

$$
a(X)+a(Y)=\left \{
\begin{array}{ll}
\frac{(n-1)^2}{4}, & \mbox{if $n$ is odd\
}\\
\frac{n^{2}}{4}-\frac n 2, &\mbox{if $n$ is even\ },
\end{array}
\right.
$$

Again $a(X,Y)\leq a(Y, X)+a(X)+a(Y)+a(Z, X)+a(Y, Z)$ fails. By
Theorem 3.2, $T$ has no $\overrightarrow{P_3}$-decomposition.

\end{proof}

\section {\large Bipartite digraph}
A graph (digraph) is bipartite if its vertex set can be partitioned
into two subsets $X$ and $Y$ so that every edge (arc) has one end in
$X$ and one end in $Y$. In this section, we give a characterization
for a bipartite digraph with $\overrightarrow{P_{3}}$-decomposition.

\begin{theorem} A bipartite digraph $D=D[X, Y]$ has a
$\overrightarrow{P_{3}}$-decomposition if and only if
$$d^{+}(X)=d^{-}(X)$$ and $$a(X_{1}, Y_{1})+a(Y_{1},X_{1})\leq
d^{+}(Y_{1})+d^{-}(X_{1})$$ for any $X_{1}\subseteq X, Y_{1}\subseteq Y$.
\end{theorem}
\begin{proof} Let $X'=A(X, Y), Y'=A(Y, X)$. Since $D=D[X, Y]$ is a bipartite digraph,
both $X'$ and $Y'$ are independent sets of $L(D)$. So,
$L(D)=L(D)[X', Y']$ is a bipartite graph. By Proposition 1.2, $D$
has a $\overrightarrow{P_{3}}$-decomposition if and only if $L(D)$
has a perfect matching.

First, we assume that $D$ has a
$\overrightarrow{P_{3}}$-decomposition. For two subsets
$X_{1}\subseteq X, Y _{1}\subseteq Y$, let $S=A(X_{1}, Y_{1})$.
Clearly, $S\subseteq X'$. One can see that
$$N_{L(D)}(S)=\partial^-(X_1)\cup \partial^+(Y_1)$$ and thus
$$|N_{L(D)}(S)|=d^{+}(Y_{1})+d^{-}(X_{1})-a(Y_{1}, X_{1}).$$

By Proposition 1.2 and Corollary 1.6, $d^{+}(X)=d^{-}(X)$ and
$$a(X_1,Y_1)=|S|\leq |N(S)|=d^{+}(Y_{1})+d^{-}(X_{1})-a(Y_{1},
X_{1}).$$ Equivalently,
$$a(X_{1}, Y_{1})+a(Y_{1},X_{1})\leq d^{+}(Y_{1})+d^{-}(X_{1})$$ for
any $X_{1}\subseteq X, Y_{1}\subseteq Y$.

Conversely, to prove that $D$ has a
$\overrightarrow{P_3}$-decomposition, it suffices to prove that
$L(D)$ has a perfect matching. For any $S\subseteq X'$, let $X_{1}$
be the set of tails of arcs in $S$ and $Y_{1}$ be the set of heads
of arcs in $S$. Clearly, $S\subseteq A(X_{1}, Y_{1})$ and
$N_{L(D)}(S)=\partial^{+}(Y_{1})\cup
\partial^{-}(X_{1})$. Since $$a(X_{1}, Y_{1})+a(Y_{1},X_{1})\leq
d^{+}(Y_{1})+d^{-}(X_{1}), $$ we conclude that

$$|N_{L(D)}(S)|=d^{+}(Y_{1})+d^{-}(X_{1})-a(Y_{1},X_{1})\geq
a(X_{1}, Y_{1})\geq |S|$$

By Corollary 1.6, $L(D)$ has a perfect matching.
\end{proof}

\section {\large Concluding remarks}

In this paper, we give a complete characterization for a tournament
and a bipartite digraph admitting a
$\overrightarrow{P_{3}}$-decomposition. However, the general case
still remains to be unsolved.

\vspace{2mm}\noindent {\bf Problem 1.} (\cite{D}) A simple
characterization of all $\overrightarrow{P_3}$-decomposable
digraphs.

\vspace{2mm} It is interesting that we established, in Theorem 3.1,
a necessary and sufficient condition for the existence of a
fractional perfect matching in $L(D)$ of a digraph $D$. Recall that
a path or a cycle which contains every vertex of a graph is called a
{\it Hamilton path} or {\it cycle} of the graph. A graph is {\it
hamiltonian} if it contains a Hamilton cycle. It is well known that
the problem of deciding whether a given graph is hamiltonian is
$\mathcal{NP}$-complete \cite{Garey}. A {\it directed Euler trail}
is a directed trail which traverses each arc of the digraph exactly
once, and a {\it directed Euler tour} is a directed tour with this
same property. A digraph is {\it eulerian} if it admits a directed
Euler tour. It is well known that a connected digraph $D$ is
eulerian if and only if $d^+(v)=d^-(v)$ for any vertex $v\in V(D)$.
R\'{e}dei \cite{Redei} proved that every tournament has a directed
Hamilton path. Camion \cite{Camion} showed that every nontrivial
strong tournament has a directed Hamilton cycle.

\begin{theorem} If a digraph $D$ of size $m\geq 3$ is eulerian, then
$L(D)$ is hamiltonian.
\end{theorem}

\begin{proof} Let $v_0, a_1, v_2, a_2, \cdots, a_m, v_0$ be an eulerian
trail, where $v_{i-1}$ and $v_i$ is the tail and the head of $a_i$,
respectively, for each $i\in \{1, \ldots, m\}$. Then $a_1, a_2,
\ldots, a_m$ is a Hamilton cycle of $L(D)$.
\end{proof}

The hamiltonian properties of line graphs have been widely studied.
Two well-known conjectures in hamiltonian graph theory are due to
Thomassen \cite{Th} and Matthews and Sumner \cite{Matt},
respectively.

\vspace{2mm}\noindent{\bf Conjecture 1.} (Thomassen \cite{Th}) Every
4-connected line graph is hamiltonian.

\vspace{2mm}\noindent{\bf Conjecture 2.} ( Matthews and Sumner
\cite{Matt}) Every 4-connected claw-free graph is hamiltonian.

Ryj\'{a}\v{c}ek \cite{R} proved that the above two conjectures are
equivalent. Very recently, Li et al. \cite{Li} investigated a class
of spanning subgraphs of a line graph.

\vspace{2mm}\noindent{\bf Definition 2.} Let $G$ be a graph. A graph
is called $SL(G)$ if

\noindent (1) it is a spanning subgraph of $L(G)$, and

\noindent (2) in this graph every vertex $e=uv$ is adjacent to at
least $\min\{d_G(u)-1, \lceil \frac 3 4 d_G(u)+\frac 1 2\rceil \}$
vertices of $E_G(u)$ and also adjacent to at least $\min\{d_G(v)-1,
\lceil \frac 3 4 d_G(v)+\frac 1 2\rceil \}$ vertices of $E_G(v)$.

\vspace{2mm} Li et al. \cite{Li} proposed the following conjecture,
and proved that it is equivalent to Conjectures 1 and 2.

\vspace{2mm}\noindent{\bf Conjecture 3.} (Li et al. \cite{Li}) Every
4-connected $SL(G)$ is hamiltonian.

\vspace{2mm} Note that $L(D)$ is a spanning subgraph of $L(G)$ for
any orientation $D$ of a graph $G$. So, if $L(D)$ is hamiltonian,
then $L(G)$ is hamiltonian. However, there exists a graph $G$ whose
line graph is hamiltonian, but $L(D)$ is nonhamiltonian for any
orientation $D$ of $G$. Such an example is shown in Figure 2.

\begin{center}
\scalebox{0.2}{\includegraphics{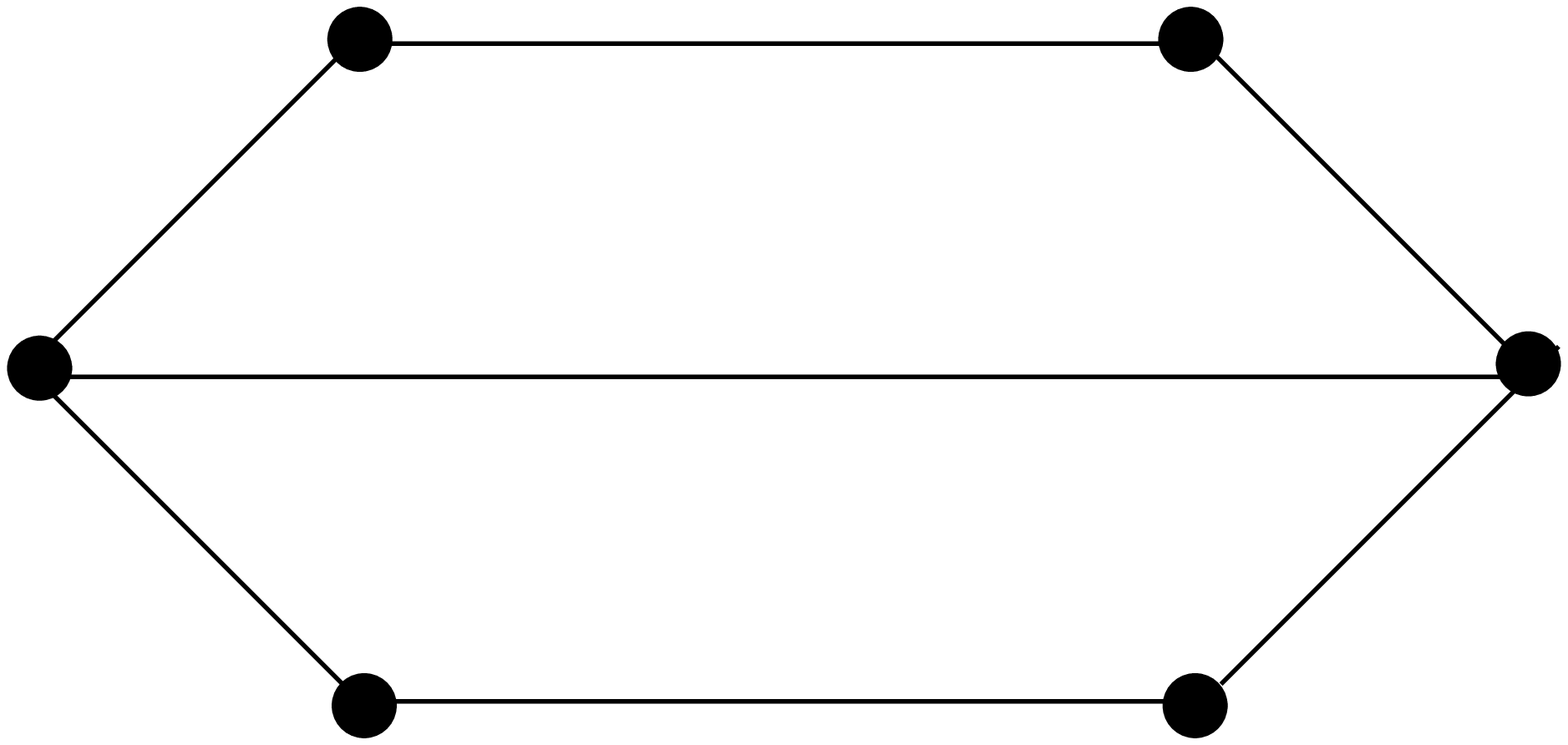}}\\
\vspace{0.5cm} Fig. 2. A graph $G$ whose $L(D)$ is nonhamiltonian
any orientation $D$ of $G$
\end{center}

It is natural to look for some other reasonable necessary or
sufficient conditions for the existence of a Hamilton cycle in
$L(D)$ for a digraph $D$. When does $L(T)$ have a Hamilton path or
Hamilton cycle for a tournament $T$ ? More generally, one considers
the following problem.

\vspace{2mm}\noindent {\bf Problem 2.} Characterize all digraphs $D$
whose line graph $L(D)$ is hamiltonian (or line graph $L(D)$ has a
2-factor or possesses some other properties).

\vspace{2mm} Beineke \cite{Be} proved that a graph $H$ is a line
graph of a graph if and only if it has no induced subgraph
isomorphic to one of the nine graphs as shown in Figure 3.

One naturally poses the following problem.

\vspace{2mm}\noindent {\bf Problem 3.} Characterize those graphs
that are line graphs of some digraph.

\vspace{2mm} It is straightforward to check that the line graph of
any strict digraph has no an induced subgraph isomorphic to $K_4\backslash e$,
which is the graph obtained from $K_4$ deleting an edge. Moreover,
the line graph of any strict asymmetric bipartite digraph has no an
induced subgraph isomorphic to $K_{3,3}\backslash e$, which is the graph
obtained from $K_{3,3}$ deleting an edge.

\begin{center}
\scalebox{0.4}{\includegraphics{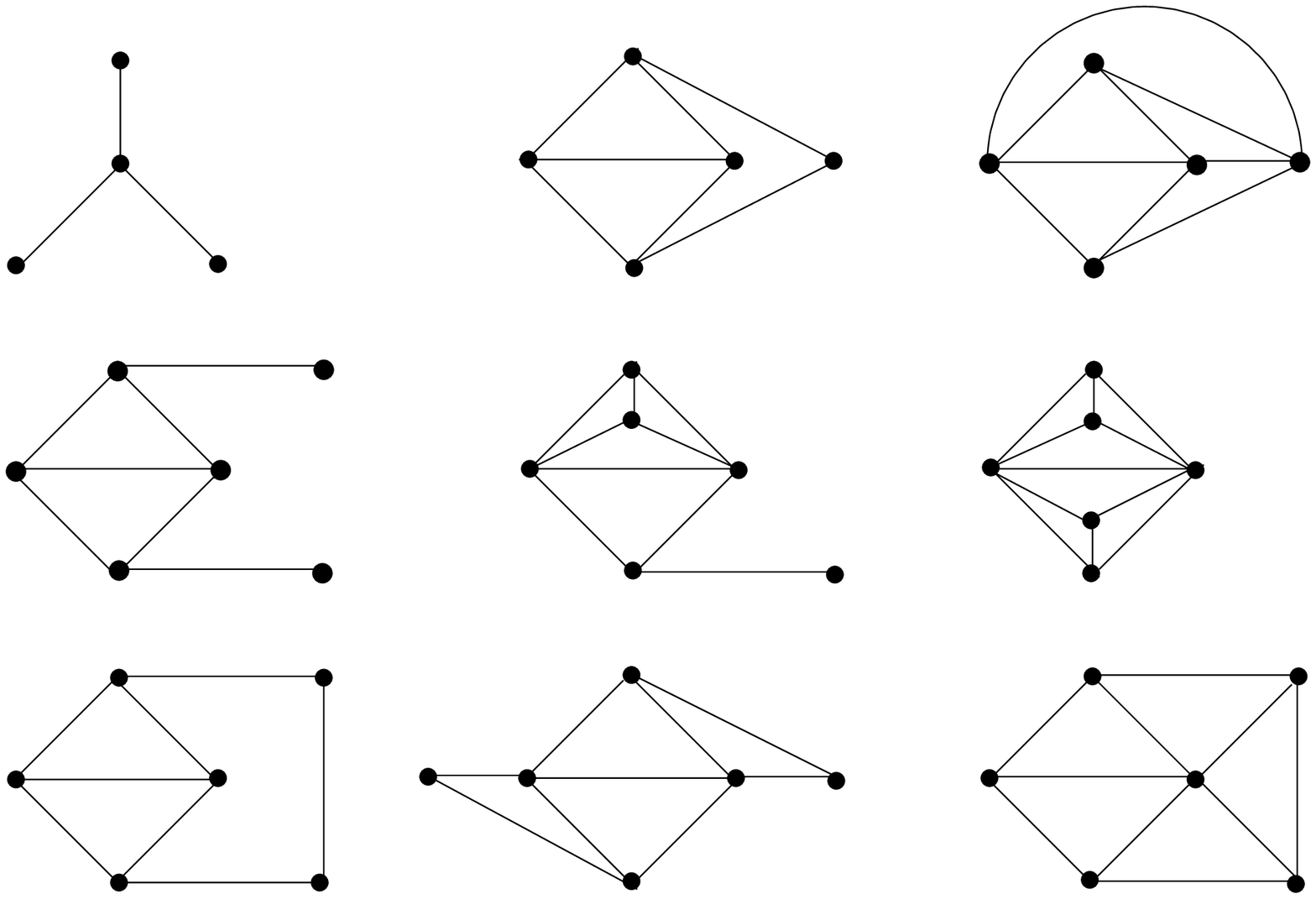}}\\
\vspace{0.5cm} Fig. 3. Nine forbidden subgraphs for a line graph
\end{center}

\noindent {\bf Problem 4.} Does there exist a forbidden subgraph
characterization for a line graph of some digraph.


\begin{thebibliography}{111}
\bibitem{Be} L.W. Beineke, \emph{Characterizations of derived graphs}, J.
Combin. Theory 9 (1970) 129-135.
\bibitem{BM} J.A. Bondy, U.S.R. Murty, \emph{Graph Theory}, Graduate Texts in Mathematics, Vol 244, Springer,
Heidelberg, 2008.
\bibitem{Camion} P. Camion, \emph{Chemins et circuits hamiltoniens des
graphes complets}, C. R. Acad. Sci. Paris 249 (1959) 2151-2152.
\bibitem{D} A.A. Diwan, \emph{$\overrightarrow{P_{3}}$-decomposition of directed graphs}, Discrete Appl. Math., http:// dx.doi.org/10.1016/j.dam.2016.01.039.
\bibitem{Garey} M.R. Garey, D.S. Johnson, \emph{Computers and
Intractability: a Guide to the Theory of
$\mathcal{NP}$-completeness}, Freeman, San Francisco, 1979.
\bibitem{H} P. Hall, \emph{On representatives of subsets}, J. London Math.
Soc. 10 (1935) 26-30.
\bibitem{K} A. Kotzig, \emph{Aus der Theorie der endlichen regul\"{a}ren Graphen dritten und vierten Grades} (Slovak), \v{C}aspis P\v{e}st. Mat. 82 (1957) 76-92.
\bibitem{Redei} L. R\'{e}dei, \emph{Ein Kombinatorischer Satz}, Acta. Litt.
Sci. Szeged 7 (1934) 39-43.
\bibitem{Li} H. Li, W. He, W. Yang, Y. Bai, \emph{Hamiltonian cycles in
spanning subgraphs of line graphs}, Discrete Appl. Math. 209 (2016)
287-295.
\bibitem{Matt} M.M. Matthews, D.P. Sumner, \emph{Hamiltonian results in
$K_{1,3}$-free graphs}, J. Graph Theory 8 (1984) 139-146.

\bibitem{R} Z. Ryj\'{a}\v{c}ek, \emph{On a closure concept in claw-free
graphs}, J. Combin. Theory Ser. B 70 (1997) 217-224.

\bibitem{Th} C. Thomassen, \emph{Reflections on graph theory}, J. Graph
Theory 10 (1986) 309-324.

\bibitem{Tut} W.T. Tutte, \emph{The factorization of linear graphs}, J. London Math. Soc. 22 (1947) 107-111.
\bibitem{Tutte} W.T. Tutte, \emph{The factor of graphs}, Can. J. Math. 4 (1952)
    314-328.
\end{thebibliography}
\end{document}